\documentclass[11pt]{article}

\usepackage{amsmath,amssymb,amsthm}
\usepackage{geometry}
\usepackage{booktabs}
\usepackage{tikz}
\usetikzlibrary{arrows.meta,positioning,calc}
\usepackage[colorlinks=true,linkcolor=blue,citecolor=blue]{hyperref}
\usepackage[capitalize,noabbrev]{cleveref}

\usepackage{amsmath,amssymb,amsthm}
\usepackage{geometry}
\usepackage{graphicx}
\usepackage{enumitem}
\usepackage[dvipsnames]{xcolor}
\usepackage[colorlinks=true,linkcolor=blue,citecolor=blue]{hyperref}
\usepackage[capitalize,noabbrev]{cleveref}

\theoremstyle{plain}
\newtheorem{theorem}{Theorem}[section]
\newtheorem{lemma}[theorem]{Lemma}

\theoremstyle{definition}
\newtheorem{definition}[theorem]{Definition}
\newtheorem{example}[theorem]{Example}

\newcommand{\R}{\mathbb R}

\newcommand{\bz}{\mathbf{z}}

\newcommand{\sV}{\mathsf{V}}
\newcommand{\sE}{\mathsf{E}}
\newcommand{\kk}{\kappa}
\newcommand{\chibp}{\chi_{\mathrm{hub}}}
\newcommand{\bv}{\mathbf{v}}
\newcommand{\bc}{\mathbf{b}}

\title{On Transformer Dynamics}
\author{Mohammad Javad Latifi Jebelli\\
  Department of Mathematics, Dartmouth College, Hanover, NH, USA\\ \\
  }

\date{}

\begin{document}
\maketitle

\begin{abstract}
We develop a geometric framework in which the token dynamics of a transformer
are modeled by a system of interacting particles on a Riemannian manifold
$\mathcal M$, the attention mechanism being encoded by a time-independent
two-body interaction law, that is, a section of the pullback bundle
$\pi_2^{*}(T\mathcal M)$ over $\mathcal M\times\mathcal M$. Within this framework
we isolate two features that a family of interaction laws must possess in order
to model language: it must realize generic nonlocal and nonreciprocal forces,
and it must parametrize vector fields on a high-dimensional manifold
efficiently. We show that both features are achieved simultaneously in a transformer model. Our main
theorem produces a finitely parametrized family of interaction laws that is \emph{universal}: it
realizes an arbitrary prescribed attention digraph. Moreover, we show that the
cost of realizing a given attention digraph is governed by two combinatorial invariants of the digraph, namely its
biclique cover number, which we identify with the least number of hubs in a
hub extension, and its hub-chromatic index.
\end{abstract}

\noindent\textbf{Mathematics Subject Classification:}
68T07;
05C20, 05C15, 05C70, 53C99, 82C22.

\section{Introduction}

There is at present no transparent way to visualize how a large language model
(LLM) produces coherent text. One may nonetheless ask for a mathematical
account of the capacity of such models to express rich structure: what is the
source of the expressive power of transformer models, and to what extent can it
be explained? The present paper is an attempt to make progress on these
questions.

Sequence generation in an LLM may be viewed as a complex particle interaction
on a high-dimensional space \cite{TR}, the interaction being governed by a rich
family of vector fields. Our central contention is that such a family must have
two essential features. First, it must be able to encode generic nonlocal and
nonreciprocal forces: a particle should be able to exert, on particles lying in
remote regions of the space, a force of arbitrary direction. Second, it must
provide an efficient parametrization of vector fields on a high-dimensional
space. The second requirement is severe. Even a coarse Fourier resolution of
functions on a space of several thousand dimensions already involves a number
of degrees of freedom far beyond what could be parametrized directly.%
\footnote{For instance, the polynomials of degree at most $10$ on
$\mathbb R^{10000}$ already form a space of dimension $\binom{10010}{10}$.}
We introduce a geometric framework for transformer models on a class of
manifolds and prove that it enjoys both features. In particular the framework
encodes generic nonlocal and nonreciprocal forces through what we call
universal families of interaction fields; and, more to the point, it represents
such families using a modest number of parameters.

Consider a manifold $\mathcal M$ and fix a discrete subset
$\mathcal V\subset\mathcal M$, the embedded vocabulary. From the dynamical point
of view, a language model constructs sequences
$x_1,x_2,\ldots\in\mathcal V$ by geometric dynamics on $\mathcal M$: given a
partially generated sequence $x_1,\ldots,x_N\in\mathcal V$, one seeks the
conditional probability
\begin{equation}\label{eq:conditional}
    \mathcal P\bigl(x_{N+1}\mid x_1,\ldots,x_N\bigr)
\end{equation}
and predicts the $(N+1)$-th token $x_{N+1}$. To build such conditional
probabilities, consider a system of $N$ interacting particles on $\mathcal M$
with initial conditions $x_1,\ldots,x_N$, evolving under a complex interaction
law; here $\mathcal M$ plays the role of the configuration space of a single
particle. Having evolved the system up to some time $T>0$, we take the terminal
position of particle $i$ to encode a probabilistic prediction of the token
$x_{i+1}$; in particular, the terminal state of the last particle yields the
prediction of $x_{N+1}$ in \eqref{eq:conditional}. This per-particle reading is
consistent with the causal structure of the dynamics introduced below, in which
particle $i$ is influenced only by the particles $z_j$ with $j\le i$.

\begin{figure}[ht]
    \centering
    \includegraphics[width=0.35\linewidth]{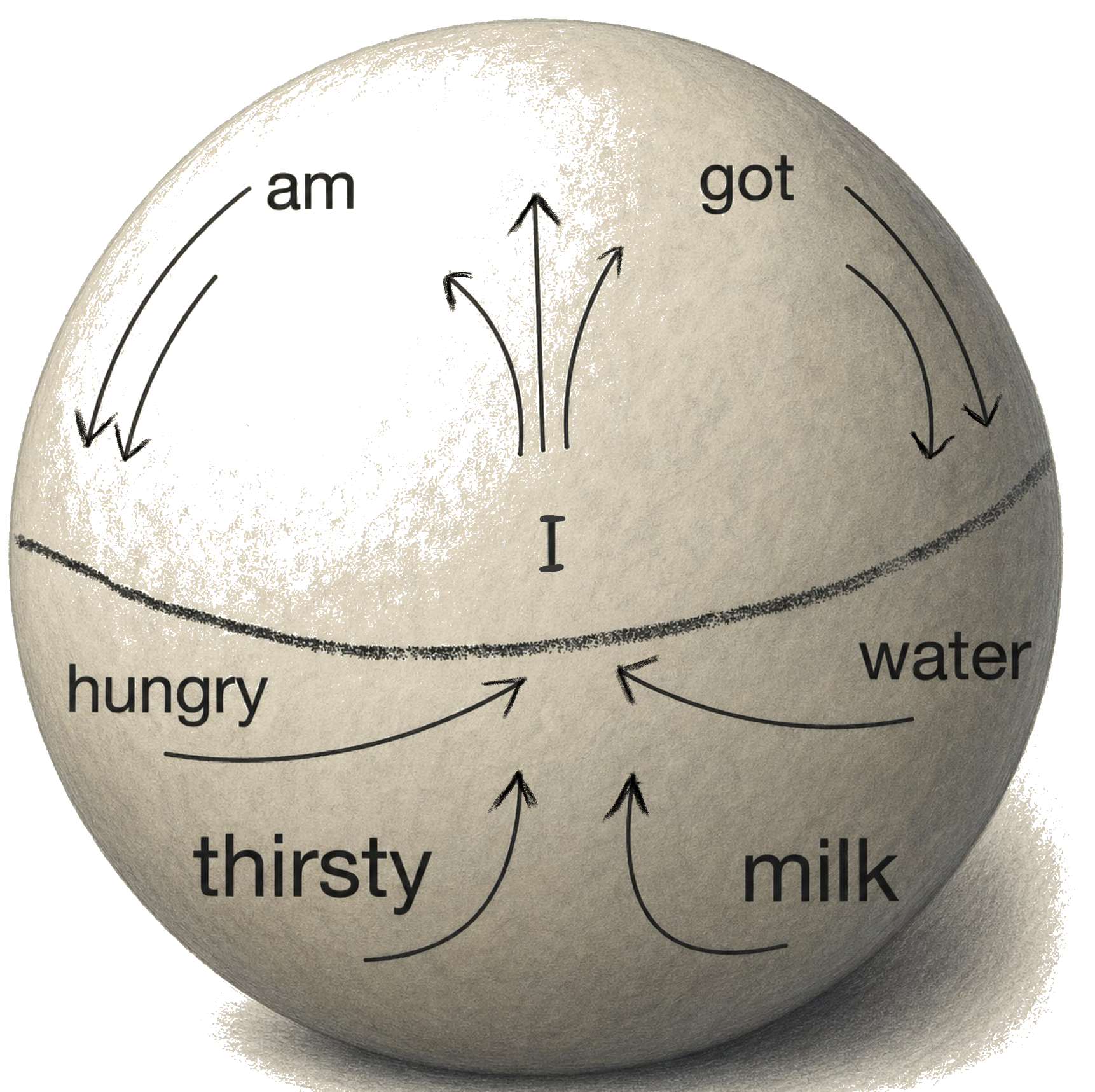}
    \caption{A toy (incomplete) picture of text generation. The ambient space $\mathcal M$
    is the two-sphere, and the finite vocabulary $\mathcal V\subset\mathcal M$
    consists of seven embedded tokens. The admissible language is generated from
    four elementary sentences---``I am hungry'', ``I am thirsty'', ``I got
    milk'', ``I got water''---sufficient for a rudimentary infant--parent
    exchange. Given the initial token ``I'', the transformer-induced dynamics
    moves the corresponding state toward the regions associated with the
    admissible continuations ``am'' and ``got''; the next token is then drawn at
    random from these. Repeating the procedure conditionally on the selected
    token yields a random discrete trajectory through $\mathcal V$, and hence a
    generated sentence.}
    \label{fig:sphere}
\end{figure}

To make the discussion precise we first define a two-body interaction. Let
$\pi_i:\mathcal M\times\mathcal M\to\mathcal M$ be the projections
$\pi_1(x,y)=x$ and $\pi_2(x,y)=y$, and let $E=\pi_2^{*}(T\mathcal M)$ be the
pullback of the tangent bundle, a vector bundle over $\mathcal M\times\mathcal M$.

\begin{definition}
A \emph{two-body interaction law} assigns to every ordered pair of points
$(x,y)\in\mathcal M\times\mathcal M$ a tangent vector
\[
u(x,y)\in T_y\mathcal M.
\]
Equivalently, a two-body interaction law is a section of $E=\pi_2^{*}(T\mathcal M)$,
and a family of interaction laws is a subset of $\Gamma(E)$.
\end{definition}

The evolution of a physical particle system is governed by Hamiltonian vector
fields on a symplectic manifold. The families of interaction vector fields
considered here are not physical in this sense; they are prescribed by vector
fields on a general manifold $\mathcal M$. The particle perspective is
nonetheless a fruitful one.

The dynamics of $N$ interacting particles with initial points
$x_1,\ldots,x_N\in\mathcal M$ is
\begin{equation}\label{eq:dynamics}
\frac{dz_i(t)}{dt}=\nu_0\bigl(z_i(t)\bigr)+\sum_{j=1}^{i}u\bigl(z_j(t),z_i(t)\bigr),
\qquad z_i(0)=x_i\in\mathcal M,\qquad i=1,\ldots,N,
\end{equation}
where $\nu_0\in\Gamma(T\mathcal M)$ is a fixed vector field (accounting for the
multilayer perceptron) and $u\in\Gamma(E)$ encodes the interaction law
(corresponding to the attention mechanism). A richer version of the interaction
is expressed through several time-dependent families of interaction laws,
\begin{equation}\label{eq:dynamics-multi}
\frac{dz_i(t)}{dt}=\nu_0^{t}\bigl(z_i(t)\bigr)
+\sum_{j=1}^{i}u^{t}_{\bz}\bigl(z_j(t),z_i(t)\bigr),
\end{equation}
where $u^{t}_{\bz}\in\Gamma(E)$ now depends on $\bz=(z_1,\ldots,z_N)$
through softmax normalization and the superscript $t$ records the time dependence. This time dependence corresponds to stacking multiple layers in a transformer model: discretizing such a time-dependent differential equation yields precisely the multilayer formulation of transformers (a general instance of which is given in \cite{neuralode}). On the other hand, combining multiple interaction laws models multi-head attention. In general, a transformer is built from such systems of
interaction laws; the aim of this paper is to understand its basic building
blocks. Specifically, we ask how a single time-independent interaction law $u$
can encode a rich attention mechanism. We do not discuss other aspects of the transformer model such as multilayers or the multilayer perceptron.

\subsection*{Attention Digraph}
An \emph{attention digraph}
is simply a digraph $(\sV,\sE)$. The underlying struture of transformers encode an enormous attention digraph that leads to a complex particle
dynamic on a high-dimensional manifold. We discuss an explicit mathematical procedure to implement such digraph optimally using a family of interaction laws (here $\sV$ is the node set of the attention graph;
it is not to be confused with the embedded vocabulary $\mathcal V\subset\mathcal M$
of the introduction). \\ 

A \emph{biclique} of $(\sV,\sE)$ is a pair $A,B\subseteq\sV$
with $A\times B\subseteq\sE$; the \emph{biclique cover number} $\bc(\sV,\sE)$ is
the least number of bicliques whose union is $\sE$.

\begin{definition}\label{def:hubextension}
A \emph{hub extension} of a digraph $(\sV,\sE)$ is a digraph
$(\widetilde\sV,\widetilde\sE)$ with hub set $H$, where
$\widetilde\sV=\sV\cup H$, such that
\begin{enumerate}
\item $(a,b)\in\sE$ if and only if there is a hub $h\in H$ with
$(a,h),(h,b)\in\widetilde\sE$; and
\item every edge of $\widetilde\sE$ has the form $(a,h)$ or $(h,a)$ for some
$h\in H$ and $a\in\sV$.
\end{enumerate}
\end{definition}

In graph-theoretic terms, a hub extension is the directed analogue of an
intersection graph for $(\sV,\sE)$---more precisely, a connection digraph---but
we adopt the hub-extension terminology as more intuitive. Each hub covers the
directed edges of a directed biclique of $(\sV,\sE)$, and, by
Lemma \ref{lem:biclique}, the least number of hubs in a hub extension of $(\sV,\sE)$
equals its biclique cover number. \\
 
A \emph{coloring} of a hub extension assigns a color to each hub so that any
two hubs sharing a source (non-hub) vertex, or sharing a target (non-hub)
vertex, receive distinct colors (or equivalently, a \emph{coloring} of a hub extension is a coloring of the edges $\widetilde\sE$
subject to the following rules: all edges incident to a common hub $h\in H$
receive the same color; and at every non-hub node any two incoming edges receive
distinct colors, as do any two outgoing edges). The chromatic index of a hub extension is then defined to be the minimum number of colors needed for such coloring. 

\begin{definition}
The \emph{hub-chromatic index} $\chibp(\sV,\sE)$ of a digraph $(\sV,\sE)$ is the
minimum, over all hub extensions of $(\sV,\sE)$, of the chromatic index of the
hub extension.
\end{definition}

 For comparison we also
record the ordinary chromatic index $\chi'(\sV,\sE)$: the least number of
colors in a coloring of the edges of $(\sV,\sE)$ itself under the rule that
two edges sharing a source or a target receive distinct colors.  We have $\chibp(\sV,\sE)\le\chi'(\sV,\sE)$. \\

\textbf{Remark.} What is the mathematical idea that supports the use of key, query, and 
value matrices in transformer models? For instance, what happens if we fix the key matrix 
in the transformer architecture to be the identity matrix? As we will see, in that case, $\chi'(\sV,\sE)$ 
becomes the essential invariant of the attention digraph instead of $\chibp(\sV,\sE)$ 
(which in practice can be significantly larger for large digraphs). Consequently, hub 
extensions play a crucial role in explaining the expressive power of transformers.\\

Now, we state two examples of digraphs and their invariants. 

\begin{example}\label{ex:triangle}
Let $\sV=\{1,2,3\}$ and $\sE=\{(1,2),(1,3),(2,3)\}$: node $1$ attends to both
$2$ and $3$, and $2$ additionally attends to $3$---the transitive tournament
on three tokens, the smallest digraph in which one token is influenced both
directly and through an intermediary.
\end{example}
 
\begin{center}
\begin{tikzpicture}[scale=1,>=Stealth,every node/.style={font=\small}]
\node[circle,draw,fill=blue!10,minimum size=7mm,inner sep=0pt] (1) at (0,2) {$1$};
\node[circle,draw,fill=blue!10,minimum size=7mm,inner sep=0pt] (2) at (-1.3,0) {$2$};
\node[circle,draw,fill=blue!10,minimum size=7mm,inner sep=0pt] (3) at (1.3,0) {$3$};
\draw[->] (1) -- (2);
\draw[->] (1) -- (3);
\draw[->] (2) -- (3);
\node at (0,2.9) {\footnotesize original digraph: $3$ edges};
\end{tikzpicture}
\hspace{1.5cm}
\begin{tikzpicture}[scale=1,>=Stealth,every node/.style={font=\small}]
\node[circle,draw,fill=blue!10,minimum size=7mm,inner sep=0pt] (1) at (0,2) {$1$};
\node[circle,draw,fill=blue!10,minimum size=7mm,inner sep=0pt] (2) at (-1.3,0) {$2$};
\node[circle,draw,fill=blue!10,minimum size=7mm,inner sep=0pt] (3) at (1.3,0) {$3$};
\node[rectangle,draw,fill=red!25,minimum size=6.5mm,inner sep=1pt] (h1) at (0.9,1.3) {$h_1$};
\node[rectangle,draw,fill=cyan!25,minimum size=6.5mm,inner sep=1pt] (h2) at (0,-0.9) {$h_2$};
\draw[->,red!70!black] (1) -- (h1);
\draw[->,red!70!black] (h1) -- (2);
\draw[->,red!70!black] (h1) -- (3);
\draw[->,cyan!70!black] (2) -- (h2);
\draw[->,cyan!70!black] (h2) -- (3);
\node at (0,2.9) {\footnotesize hub extension: $2$ hubs, $2$ colors};
\end{tikzpicture}
\end{center}
 
For the digraph of Example \ref{ex:triangle},
\[
\bc(\sV,\sE)=2,\qquad \chibp(\sV,\sE)=2,\qquad \chi'(\sV,\sE)=2.
\]

\begin{proof}
\emph{Biclique cover, upper bound.} The bicliques $\{1\}\times\{2,3\}$ and
$\{2\}\times\{3\}$ lie in $\sE$ and their union is $\sE$, so
$\bc(\sV,\sE)\le2$.
 
\emph{Biclique cover, lower bound.} Suppose a single biclique $A\times
B\subseteq\sE$ covered all three edges. Covering $(1,2)$ forces $1\in A$,
$2\in B$; covering $(1,3)$ forces $1\in A$, $3\in B$; covering $(2,3)$ forces
$2\in A$, $3\in B$. Together $A\supseteq\{1,2\}$, $B\supseteq\{2,3\}$, so
$2\in A\cap B$ and the biclique condition would force $(2,2)\in\sE$---but
there are no self-loops. Hence no single biclique covers $\sE$, and
$\bc(\sV,\sE)\ge2$. Combining, $\bc(\sV,\sE)=2$.
 
\emph{Hub-chromatic index, lower bound.} Suppose some hub extension of
$(\sV,\sE)$ used only one color. Then every vertex has at most one outgoing
and at most one incoming edge in the extension. Node $1$ has out-degree $2$
in $\sE$ (to $2$ and to $3$), so both edges must be realized through $1$'s
unique outgoing hub $h$: $1\in A_h$, $\{2,3\}\subseteq B_h$. Node $3$ has
in-degree $2$ in $\sE$ (from $1$ and from $2$), so both edges must be
realized through $3$'s unique incoming hub $h'$: $\{1,2\}\subseteq A_{h'}$,
$3\in B_{h'}$. Since edge $(1,3)$ can only be realized via $1$'s unique
outgoing hub, $h'=h$; but then $A_h\supseteq\{1,2\}$, $B_h\supseteq\{2,3\}$,
so $2\in A_h\cap B_h$, forcing $(2,2)\in\sE$ by the hub axiom---impossible.
Hence $\chibp(\sV,\sE)\ge2$.
 
\emph{Hub-chromatic index, upper bound.} The extension with hubs
$h_1=(\{1\},\{2,3\})$ and $h_2=(\{2\},\{3\})$ realizes $\sE$; coloring $h_1$
and $h_2$ differently is trivially valid (each color class has only one
hub), so $\chibp(\sV,\sE)\le2$. Hence $\chibp(\sV,\sE)=2$.
 
\emph{Ordinary chromatic index.} Node $1$'s two outgoing edges force
$\chi'\ge2$; the coloring $(1,2)\mapsto1,(1,3)\mapsto2,(2,3)\mapsto1$ is
valid (no two edges sharing a source or a target receive equal colors), so
$\chi'(\sV,\sE)=2$.
\end{proof}

\begin{example}\label{ex:bipartite}
Let $\sV=A\sqcup B$ with $A=\{a_1,a_2,a_3\}$, $B=\{b_1,b_2,b_3,b_4\}$, and
$\sE=A\times B$: every node of $A$ attends to every node of $B$, a fully
dense, ``global'' attention block on seven tokens.
\end{example}
 
\begin{center}
\begin{tikzpicture}[scale=0.85,>=Stealth,every node/.style={font=\small}]
\foreach \i/\y in {1/3,2/2,3/1}
  \node[circle,draw,fill=blue!10,minimum size=7mm,inner sep=0pt] (a\i) at (0,\y) {$a_{\i}$};
\foreach \j/\y in {1/3.5,2/2.5,3/1.5,4/0.5}
  \node[circle,draw,fill=green!10,minimum size=7mm,inner sep=0pt] (b\j) at (3,\y) {$b_{\j}$};
\foreach \i in {1,2,3}
  \foreach \j in {1,2,3,4}
    \draw[->,thin,gray] (a\i) -- (b\j);
\node at (1.5,4.1) {\footnotesize original digraph: $12$ edges};
\end{tikzpicture}
\hspace{1.2cm}
\begin{tikzpicture}[scale=0.85,>=Stealth,every node/.style={font=\small}]
\foreach \i/\y in {1/3,2/2,3/1}
  \node[circle,draw,fill=blue!10,minimum size=7mm,inner sep=0pt] (a\i) at (0,\y) {$a_{\i}$};
\foreach \j/\y in {1/3.5,2/2.5,3/1.5,4/0.5}
  \node[circle,draw,fill=green!10,minimum size=7mm,inner sep=0pt] (b\j) at (3,\y) {$b_{\j}$};
\node[rectangle,draw,fill=orange!30,minimum size=7mm,inner sep=1pt] (h) at (1.5,2) {$h$};
\foreach \i in {1,2,3}
  \draw[->] (a\i) -- (h);
\foreach \j in {1,2,3,4}
  \draw[->] (h) -- (b\j);
\node at (1.5,4.1) {\footnotesize hub extension: $1$ hub, $7$ edges};
\end{tikzpicture}
\end{center}

For this digraph,
\[
\bc(\sV,\sE)=1,\qquad \chibp(\sV,\sE)=1,\qquad \chi'(\sV,\sE)=4.
\]
 
\begin{proof}
\emph{Biclique cover.} The biclique $A\times B$ is contained in $\sE$ by
definition and its union is all of $\sE$, so $\bc(\sV,\sE)\le1$; since
$\sE\ne\emptyset$, $\bc(\sV,\sE)\ge1$. Hence $\bc(\sV,\sE)=1$, achieved by the
extension with a single hub $h$ and edges $(a_i,h)$, $(h,b_j)$ for all $i,j$.
 
\emph{Hub-chromatic index.} A hub extension with a single hub has no other
hub to conflict with, so one color suffices: $\chibp(\sV,\sE)\le1$. As
$\chibp\ge1$ for any nonempty digraph, $\chibp(\sV,\sE)=1$.
 
\emph{Ordinary chromatic index.} Viewed as a bipartite graph, every node of
$A$ has degree $4$ and every node of $B$ has degree $3$, so $\Delta=4$. By
König's edge-coloring theorem, bipartite graphs have $\chi'=\Delta$, so
$\chi'(\sV,\sE)=4$; concretely, coloring $(a_i,b_j)$ by $(i+j)\bmod 4$ gives a
valid $4$-coloring, and $3$ colors cannot suffice since $a_1$ alone has four
outgoing edges requiring four distinct colors.

\end{proof}

\paragraph{Graph theory literature .} The invariants above sit inside a well-developed combinatorial theory.
Covering the edges of a graph by bicliques descends from the Graham--Pollak
theorem, which shows that the edges of $K_n$ cannot be partitioned into fewer
than $n-1$ complete bipartite subgraphs~\cite{GrahamPollak1971}; the associated
biclique \emph{partition} number $\mathrm{bp}$ admits sharp algebraic lower
bounds through eigenvalue and Hermitian-rank arguments, developed by Kratzke,
Reznick, and West~\cite{KratzkeReznickWest1988} and by Gregory, Shader, and
Watts~\cite{GregoryShaderWatts1999}. Our $\bc(\sV,\sE)$ is the directed
\emph{cover} analogue and satisfies $\bc\le\mathrm{bp}$; note that such rank
lower bounds need not persist for covers. The hub extension of
Definition~\ref{def:hubextension} is the directed biclique analogue of a
set-intersection representation in the sense of Erd\H{o}s, Goodman, and
P\'osa~\cite{ErdosGoodmanPosa1966}, and Lemma~\ref{lem:biclique} identifies the
least number of hubs with $\bc(\sV,\sE)$. Since the conflict rule for the
ordinary chromatic index $\chi'(\sV,\sE)$ forbids only a shared source or a
shared target, that index is the chromatic index of a bipartite conflict graph;
K\H{o}nig's edge-coloring theorem~\cite{Konig1916} therefore gives the exact
value $\chi'(\sV,\sE)=\max(\Delta^{+},\Delta^{-})$, sharper than the general
$\Delta+1$ bound of Vizing~\cite{Vizing1964}. Finally, one always has
$\chibp(\sV,\sE)\le\chi'(\sV,\sE)$, while $\chi'$ may exceed $\chibp$ without
bound---already for a single directed biclique $A\times B$ one finds
$\chi'=\max(|A|,|B|)$ yet $\chibp=1$---so fixing the key matrix to the identity
forfeits the compression that hub extensions provide and forces reliance on the
larger invariant $\chi'$.

\subsection*{Universal Family of Interaction Laws}

Our goal is to construct a rich network of nonlocal, nonreciprocal interactions
$u\in\Gamma(E)$ (or $u^{t}_{\bz}$) admitting an efficient parametrization,
one that generalizes---and in part explains---the transformer. We consider
interaction laws of the form $u=\beta a$, with attention weight
$\beta\in C^{\infty}(\mathcal M\times\mathcal M)$ and $a\in\Gamma(E)$ satisfying
$\|a(x,y)\|_{T_y\mathcal M}=1$ everywhere (if $\mathcal M$ is not parallelizable,
this may fail on a set of arbitrarily small measure). Thus the strength of the
interaction is dictated by $\beta$.

\begin{definition}
A \emph{finite-rank interaction law} is of the form $u=\beta a$ where
$\beta(x,y)$ is a finite-rank tensor in
$C^{\infty}(\mathcal M)\otimes C^{\infty}(\mathcal M)\subset
C^{\infty}(\mathcal M\times\mathcal M)$. The \emph{rank} of $u=\beta a$ is the
rank of $\beta$ as a tensor in
$C^{\infty}(\mathcal M)\otimes C^{\infty}(\mathcal M)$.
\end{definition}

A strong interaction law at $(x_0,y_0)$ indicates that a particle at $x_0$ strongly
influences a particle at $y_0$ through $u(x_0,y_0)$. \\

\begin{definition}\label{def:strong}
An interaction law $u\in\Gamma(E)$ is \emph{strong} at
$(x_0,y_0)\in\mathcal M\times\mathcal M$ if
$\beta(x,y)=\|u(x,y)\|_{T_y\mathcal M}$ attains a local maximum at
$(x_0,y_0)$ and is not constant on any neighborhood of $(x_0,y_0)$; in particular,
there is a neighborhood of $(x_0,y_0)$ on which
\[
\|u(x,y)\|_{T_y\mathcal M}\le\|u(x_0,y_0)\|_{T_{y_0}\mathcal M}.
\]
\end{definition}

In our constructions the attention weight $\beta$ is nonnegative, being a sum of
products of the nonnegative kernel $\kk$ introduced below, so that
$\beta=|\beta|$; and the family $a_\theta$ satisfies
$\|a_\theta(x,y)\|_{T_y\mathcal M}=1$ near the point of interest
(see Lemma \ref{lem:isotropicfamily}). Consequently
$\|u(x,y)\|_{T_y\mathcal M}=|\beta(x,y)|$ locally, and whenever $|\beta|$ attains
a local maximum at $(x_0,y_0)$ and is nonconstant nearby, the law $u=\beta\,a_\theta$
is strong at $(x_0,y_0)$.

 An embedding of
nodes $\alpha:\sV\to\mathcal M$ induces an embedding $\widetilde\alpha$ of edges
into $\mathcal M\times\mathcal M$, sending $(a,b)\in\sE$ to
$(\alpha(a),\alpha(b))$. A function in $C^{\infty}(\mathcal M\times\mathcal M)$
that attains its maximum exactly on the image $\widetilde\alpha(\sE)$ thus
encodes the structure of $(\sV,\sE)$. Throughout, we fix a separation scale
$r_0>0$ and consider only \emph{admissible} embeddings: injective embeddings
$\alpha:\sV\to\mathcal M$ whose images are pairwise $r_0$-separated, so that
$d(\alpha(v),\alpha(w))>r_0$ for all distinct $v,w\in\sV$. Since $\sV$ is finite
and $\mathcal V$ discrete this is no loss of generality, and unless stated
otherwise every embedding below is assumed admissible.

\begin{definition}
Let $\alpha:\sV\to\mathcal M$ be an embedding of a digraph $(\sV,\sE)$. A
\emph{density cover} of $\alpha$ is a collection of smooth functions
$\beta_1,\ldots,\beta_h\in C^{\infty}(\mathcal M\times\mathcal M)$ such that for
every $(a,b)\in\sE$ there is an index $1\le j\le h$ for which $\beta_j$ attains a
local maximum at $(\alpha(a),\alpha(b))$ and is not constant on any neighborhood
of $(\alpha(a),\alpha(b))$.
\end{definition}

The embedding $\alpha$ above is typically chosen to meet criteria independent of
the attention graph. We say that a family of interaction laws
$\mathcal I\subset\Gamma(E)$ \emph{implements} an attention graph $(\sV,\sE)$ if
there are $u_1,\ldots,u_h\in\mathcal I$ whose norms
$\beta_j=\|u_j\|\in C^{\infty}(\mathcal M\times\mathcal M)$ form a density cover
of $(\sV,\sE)$.

\begin{definition}\label{def:universal}
Let $\mathcal M$ be a Riemannian manifold. A family of interaction laws
$\mathcal I\subset\Gamma(E)$ is \emph{universal} if
\begin{enumerate}
\item there is a constant $C>0$ such that every attention digraph $(\sV,\sE)$
with $|\sV|<C$ has a density cover in $\mathcal I$, for every admissible
embedding $\alpha:\sV\to\mathcal M$; and
\item for every $x_0,y_0\in\mathcal M$ and every $v_0\in T_{y_0}\mathcal M$ there
is $u\in\mathcal I$ that is strong at $(x_0,y_0)$ and satisfies
$u(x_0,y_0)=c\,v_0$ for some constant $c>0$.
\end{enumerate}
\end{definition}

\textbf{Remark.} What are mathematical problems that can be solved effectively by a transformer model? 
Consider the following concrete example involving digraphs. Let $p \in (0,1)$ be fixed, and 
let $N$ and $n$ be positive integers. We construct a random hub extension by independently 
connecting each of the $N$ non-hub nodes to each of the $n$ hubs with probability $p$. 
This random hub extension induces a unique random digraph $(\sV,\sE)$ on $N$ nodes (replacing hubs with bicliques). 
The problem is to find a parametrized family of functions that can represent a density 
cover of any such random digraph $(\sV,\sE)$ using only a small number of parameters. A similar problem can be stated in terms of family of interaction laws implementing a random digraph. \\

\subsection*{Main Theorem}

\begin{theorem}\label{thm:main}
Let $\mathcal M$ be a smooth Riemannian manifold.
\begin{enumerate}[label=\textup{(\roman*)}]
\item There exists a finitely parametrized family of interaction laws
$\mathcal I\subset\Gamma(E)$ that is universal.
\item Let $(\sV,\sE)$ be an attention graph with an admissible embedding
$\alpha:\sV\to\mathcal M$, and let $(\widetilde\sV,\widetilde\sE)$ be a colored
hub extension of $(\sV,\sE)$ with $n=|H|$ hub nodes and $h$ colors. Then there
is a density cover $\beta_1,\ldots,\beta_h\in C^{\infty}(\mathcal M\times\mathcal M)$
of $\alpha$ with $\operatorname{rank}(\beta_j)\le n$ for $1\le j\le h$.
\end{enumerate}
\end{theorem}

Minimizing over all hub extensions of $(\sV,\sE)$ and their colorings, the least
value of $n$ that occurs is the biclique cover number of $(\sV,\sE)$ (by Lemma
\ref{lem:biclique}), and the least value of $h$ is the hub-chromatic index
$\chibp(\sV,\sE)$. In general the two minima are attained by different hub
extensions and need not be realized simultaneously; part~(ii) therefore fixes a
single colored hub extension and reports the corresponding pair $(n,h)$.

The paper thus extends the transformer to an abstract geometric setting, with
several benefits. It gives an account of language models that is both
mathematically natural and accessible; it extends transformer-like models beyond
vector spaces to homogeneous spaces, allowing, for instance, analogues in which
tokens are points of a projective space or elements of a Grassmannian; and it
furnishes a rigorous framework in which statements about such models can be
formulated and proved.

\paragraph{Organization.}
The remainder of the introduction reviews related work, placing our single,
time-independent interaction law within the dynamical-systems view of
Transformers and the literature on the expressive power of attention.
Section~\ref{sec:main} then establishes the main result: \Cref{lem:biclique}
identifies the least number of hubs in a hub extension with the biclique cover
number; \Cref{lem:isotropicfamily} constructs, on any Riemannian manifold, a
finite-dimensional isotropic family of vector fields realizing an arbitrary
prescribed tangent direction with unit norm near a given point; and these are
assembled, through a compactly supported kernel $\kk$, into a proof of
\Cref{thm:main}. Its two parts establish universality of the family
$\mathcal I$ (part~(i)) and the rank bound $\operatorname{rank}\beta_j\le n$
governed by a colored hub extension (part~(ii)). Finally, Section~3 specializes
the framework to the standard key--query--value Transformer on $\R^{d}$---which,
after normalization, is naturally regarded as living on a high-dimensional
sphere.

\subsection{Related Work}

Our framework belongs to a now-substantial body of work viewing deep networks,
and Transformers in particular, through the lens of dynamical systems. Residual
networks can be read as forward-Euler discretizations of continuous-time
ordinary differential equations \cite{weinanE,haberruthotto}, a perspective made
precise in the infinite-depth limit by neural ordinary differential equations
\cite{neuralode}. For Transformers specifically, Lu et al.\ \cite{lu} read the
architecture as a multi-particle dynamical system, and Geshkovski et al.\
\cite{TR} develop a systematic theory of self-attention as an interacting
particle system, building on the original Transformer of Vaswani et al.\
\cite{vaswani}. Our evolution equations place a single, time-independent
interaction law within this lineage: rather than studying the flow of a full,
time-dependent multi-head attention stack, we isolate the expressive power of
one interaction law and ask which attention structures it can realize.

A second line of work concerns the expressive power of attention. Yun et al.\
\cite{yun-ua} prove that Transformers are universal approximators of continuous
permutation-equivariant sequence-to-sequence functions on a compact domain, and
of arbitrary continuous such functions once positional encodings are added,
separating the role of the self-attention layers, which realize contextual maps,
from that of the feed-forward layers. Closer to the combinatorial content of the
present paper, Yun et al.\ \cite{yun-sparse} show that sparse attention with only
$O(n)$ connections per layer already suffices for universal approximation, the
guarantee being expressed through connectivity conditions on the attention
pattern. Our results address a complementary question: rather than approximating
a target function, we ask how efficiently a prescribed attention digraph can be
\emph{realized}, and we frame the cost directly in terms of that graph, through
hub extensions, the biclique cover number, and the hub-chromatic index. We are
not aware of a comparable graph-theoretic accounting of attention expressivity
in this literature.

Finally, our choice to work on a general manifold is motivated by the geometric
and optimal-transport structure of the interacting-particle picture. Geshkovski
et al.\ \cite{clusters} show that, for time-independent weights, tokens cluster
toward limiting configurations as $t\to\infty$, that the type of limiting object
is governed by the spectrum of the value matrix, and that in one dimension the
self-attention matrix converges to a low-rank Boolean matrix; the fine,
metastable structure of this convergence is analyzed further in \cite{metastability}.
These dynamics are naturally posed on the sphere, and their analysis draws on
connections to the Kuramoto model of coupled oscillators \cite{kuramoto} and to
Wasserstein gradient flows. The latter connection is made explicit by Sander et
al.\ \cite{sinkformers}, who replace the softmax normalization by Sinkhorn's
algorithm to obtain doubly-stochastic attention, and show that the resulting
iterations form a discretized Wasserstein gradient flow, and a heat diffusion in
an infinite-sample limit. The value matrix, kernel, and normalization appearing
in our construction are the manifold analogues of these objects, which is one
reason the abstract geometric setting is natural.

\section{Main Result}\label{sec:main}


\begin{lemma}[Hub extension and biclique cover]\label{lem:biclique}
For any digraph $(\sV,\sE)$, the least number of hubs in a hub extension equals
the biclique cover number of $(\sV,\sE)$.
\end{lemma}

\begin{proof}
Write $\operatorname{bc}(\sV,\sE)$ for the biclique cover number, the least
number of directed bicliques needed to cover $\sE$; by a directed biclique we
mean a product $A\times B\subseteq\sE$. We prove both inequalities.

\emph{Hub extension $\Rightarrow$ biclique cover.} Let
$(\widetilde\sV,\widetilde\sE)$ be a hub extension with hub set
$H=\{h_1,\ldots,h_n\}$, and set
\[
A_i:=\{v\in\sV:(v,h_i)\in\widetilde\sE\},\qquad
B_i:=\{v\in\sV:(h_i,v)\in\widetilde\sE\}.
\]
By the hub-extension axiom, every $(a,b)\in\sE$ is realized by some hub, so
$(a,b)\in A_i\times B_i$ for some $i$. Hence
$\{A_1\times B_1,\ldots,A_n\times B_n\}$ covers $\sE$ and
$\operatorname{bc}(\sV,\sE)\le n$.

\emph{Biclique cover $\Rightarrow$ hub extension.} Conversely, let
$\{A_1\times B_1,\ldots,A_n\times B_n\}$ be a biclique cover of $\sE$. Put
$\widetilde\sV=\sV\cup H$ with $H=\{h_1,\ldots,h_n\}$ and
\[
\widetilde\sE:=\{(v,h_i):v\in A_i\}\cup\{(h_i,v):v\in B_i\},
\qquad i\in\{1,\ldots,n\}.
\]
Then (i) every $(a,b)\in\sE$ lies in some $A_i\times B_i$, so $(a,h_i)$ and
$(h_i,b)$ lie in $\widetilde\sE$; conversely, if
$(v,h_i),(h_i,b)\in\widetilde\sE$ then $(v,b)\in A_i\times B_i\subseteq\sE$, so
no edges outside $\sE$ are realized. Moreover (ii) every edge of $\widetilde\sE$
has the form $(v,h)$ or $(h,v)$. Thus $(\widetilde\sV,\widetilde\sE)$ is a hub
extension with $n$ hubs, whence $\bc(\sV,\sE)\ge n$ after
minimizing over hub extensions. Combining the two inequalities gives the claim.
\end{proof}

\begin{lemma}\label{lem:isotropicfamily}
Let $(\mathcal M,g)$ be a smooth Riemannian manifold. Then there is a
finite-dimensional family of vector fields
\[
\bv:\Theta\to\Gamma(T\mathcal M)
\]
such that for every $y_0\in\mathcal M$ and every
$0\ne v_0\in T_{y_0}\mathcal M$ there is $\theta\in\Theta$ with
\[
\bv(\theta)_{y_0}=\operatorname{const}\cdot v_0
\qquad\text{and}\qquad
\|\bv(\theta)_y\|_g=1
\]
for all $y$ in a neighborhood of $y_0$.
\end{lemma}

\begin{proof}
Fix a smooth embedding $\iota:\mathcal M\hookrightarrow\R^{K}$ for some finite
$K$. Identify $T_p\mathcal M$ with its image in $\R^{K}$, and let
$\Pi_p:\R^{K}\to T_p\mathcal M$ be the Euclidean orthogonal projection, so that
$p\mapsto\Pi_p$ is smooth. For $a\in\R^{K}$ set
\[
X_a(p):=\Pi_p(a).
\]
Then $X_a\in\Gamma(T\mathcal M)$, and the fields $X_a$ span each tangent space
pointwise.

Choose $\chi\in C^{\infty}([0,\infty);[0,1])$ with $\chi\equiv1$ on $[0,1]$ and
$\chi<1$ on $(1,\infty)$. Let $\Theta$ be the set of triples
$(p,a,r)\in\mathcal M\times\R^{K}\times(0,\infty)$ with
\[
X_a(y)\ne0\qquad\text{whenever}\qquad|\iota(y)-\iota(p)|\le r.
\]
For $\theta=(p,a,r)\in\Theta$ define
\[
\bv(\theta)_y:=
\frac{X_a(y)}
{\Bigl(\|X_a(y)\|_g^2+\bigl(1-\chi(|\iota(y)-\iota(p)|^2/r^2)\bigr)^2\Bigr)^{1/2}}.
\]
The denominator never vanishes: on $\{|\iota(y)-\iota(p)|\le r\}$ this follows
from the definition of $\Theta$, and elsewhere from $\chi<1$. Hence
$\bv(\theta)$ is a smooth vector field.

Now fix $y_0\in\mathcal M$ and $0\ne v_0\in T_{y_0}\mathcal M$, and put
$a:=v_0/\|v_0\|_g\in T_{y_0}\mathcal M\subset\R^{K}$, so that $X_a(y_0)=a\ne0$.
By continuity $X_a$ is nonvanishing on some neighborhood $U$ of $y_0$, and
choosing $r>0$ small enough gives $|\iota(y)-\iota(y_0)|\le r\Rightarrow y\in U$;
thus $\theta=(y_0,a,r)\in\Theta$. For $y$ near $y_0$ the cutoff term equals $1$,
so
\[
\bv(\theta)_y=\frac{X_a(y)}{\|X_a(y)\|_g},
\]
whence $\|\bv(\theta)_y\|_g=1$ near $y_0$, while at $y_0$,
\[
\bv(\theta)_{y_0}=\frac{X_a(y_0)}{\|X_a(y_0)\|_g}=a=\frac{v_0}{\|v_0\|_g}.
\]
This proves the claim.
\end{proof}

\subsection*{Proof of the main theorem}

The idea is to build families of interaction laws from a vector bundle $P$ over
$\mathcal M\times\mathcal M$ (equivalently, from vector-valued functions on
$\mathcal M$), taking the attention weights to be pairings of sections,
$\langle\phi,\psi\rangle$ with $\phi,\psi\in\Gamma(P)$. For the construction to
be efficient on a high-dimensional manifold, we want the fiber dimension of $P$
to be far smaller than $\dim\mathcal M$ while still implementing a general
attention digraph; this is the content of part~(ii).

\begin{proof}[Proof of \cref{thm:main}]
Throughout, let $\kk:\mathcal M\times\mathcal M\to\R$ be a symmetric nonnegative
function $\kk\in C^{\infty}(\mathcal M\times\mathcal M)$ such that:
\begin{enumerate}
\item $\kk$ attains its local maxima on the diagonal
$\Delta_{\mathcal M}=\{(x,x):x\in\mathcal M\}$, with $\kk(x,x)=1$;
\item $\kk$ is strictly radially decreasing with respect to the Riemannian distance $d$:
if $d(x,y)\le d(x,y')$ then $\kk(x,y')\le\kk(x,y)$;
\item $\kk$ has \emph{effective radius} $r$: for every $p\in\mathcal M$ the
function $\kk(\cdot,p)$ is supported in the ball $B_r(p)$, i.e.\ $\kk(x,p)=0$
whenever $d(x,p)\ge r$. We fix $r<r_0$, with $r_0$ the separation scale of
admissible embeddings.
\end{enumerate}
Given a digraph $(\sV,\sE)$ and an admissible embedding
$\alpha:\sV\to\mathcal M$, distinct nodes satisfy
$d(\alpha(v),\alpha(w))>r_0>r$, so property~(3) gives
$\kk(\alpha(v),\alpha(w))=0$ for all distinct $v,w\in\sV$; in particular
$\kk(\alpha(a),\alpha(b))=0$ for every edge $(a,b)\in\sE$, and the balls
$B_r(\alpha(v))$, $v\in\sV$, are pairwise disjoint.

\medskip
\noindent\emph{Part (i).}
Let $a_\theta\in\Gamma(E)$ be defined by $a_\theta(x,\cdot)=\bv(\theta)$, where
$\bv(\theta)$ is the vector field of Lemma \ref{lem:isotropicfamily}. For a fixed
$N\in\mathbb N$ set
\begin{equation}\label{eq:family}
\mathcal I=\Bigl\{u\in\Gamma(E):\,
u(x,y)=\sum_{i=1}^{m}\kk(x_i,x)\,\kk(y_i,y)\,a_{\theta_i}(x,y),\;
m<N,\ \theta_i \in\Theta,\ (x_i,y_i)\in\mathcal M\times\mathcal M\Bigr\}.
\end{equation}
This family is finitely parametrized, by $x_i,y_i$ and $\theta_i\in\Theta$. Let
$(\sV,\sE)$ be an attention digraph with
$\sE=\{(a_1,b_1),\ldots,(a_m,b_m)\}\subset\sV\times\sV$ and an admissible
embedding $\alpha:\sV\to\mathcal M$, and set
\[
u(x,y)=\sum_{k=1}^{m}\kk(\alpha(a_k),x)\,\kk(\alpha(b_k),y)\,a_{\theta_k}(x,y),
\qquad\theta\in\Theta.
\]
Then $u\in\mathcal I$. By the separation established above, at the point
$(\alpha(a_i),\alpha(b_i))$ every summand with index $k\ne i$ vanishes---each
such edge differs from $(a_i,b_i)$ in at least one endpoint, and $\kk$ vanishes
between distinct nodes so
\[
u(\alpha(a_i),\alpha(b_i))=a_{\theta_i}(\alpha(a_i),\alpha(b_i)),
\]
and $\|u(x,y)\|$ attains a local maximum at each such $(\alpha(a_i),\alpha(b_i))$. We assume that $\theta_i$ is chosen so that $\|a_{\theta_i}\|\equiv1$ near $\alpha(b_i)$, and strength follows from the remark after Definition
\ref{def:strong}.

Now let $x_0,y_0\in\mathcal M$ and $0\ne v_0\in T_{y_0}\mathcal M$ be given, and
consider $u(x,y)=\beta(x,y)\,a_\theta(x,y)$ with $\beta(x,y)=\kk(x,x_0)\kk(y,y_0)$.
By Lemma \ref{lem:isotropicfamily} choose $\theta$ with $\|a_\theta\|\equiv1$ near
$y_0$ and $a_\theta(x_0,y_0)=\bv(\theta)_{y_0}=\operatorname{const}\cdot v_0$.
Then $\|u\|=\beta$ near $(x_0,y_0)$; since
$\beta=\kk(\cdot,x_0)\kk(\cdot,y_0)$ attains a local maximum at $(x_0,y_0)$ and
is nonconstant nearby, $u$ is strong at $(x_0,y_0)$, and moreover
$u(x_0,y_0)=a_\theta(x_0,y_0)=\operatorname{const}\cdot v_0$. A digraph with
$|\sV|<C$ has fewer than $C^2$ edges, so taking $N>C^2$ ensures $m=|\sE|<N$ and
$u\in\mathcal I$; this determines the constant $C$ in Definition \ref{def:universal}(1).
Hence $\mathcal I$ is universal.

\medskip
\noindent\emph{Part (ii).}
First observe that if $\phi,\psi\in C^{\infty}(\mathcal M;\R^{n})$ then
$\beta(x,y)=\langle\phi(x),\psi(y)\rangle$ has rank at most $n$. Let
$(\widetilde\sV,\widetilde\sE)$ be a colored hub extension of $(\sV,\sE)$ with
$\widetilde\sV=\sV\cup H$, where $H$ is the hub set, $n=|H|$, and the coloring
uses $h$ colors. Identify the real-valued functions on $H$ with $\R^{n}$ via an
orthonormal basis $\{e_a:a\in H\}$.

Fix a color $1\le j\le h$. For $v\in\sV$, let $s(v)\in H$ (resp.\ $t(v)\in H$) be
the unique source (resp.\ target) hub of $v$ among the color-$j$ edges;
uniqueness follows from the coloring rule, since at most one color-$j$ edge
leaves, and at most one enters, each non-hub node. If $v$ has no outgoing
(resp.\ incoming) color-$j$ edge, set $e_{s(v)}:=0$ (resp.\ $e_{t(v)}:=0$), so
that the sums below are unambiguous. Set
$\beta_j(x,y)=\langle\phi_j(x),\psi_j(y)\rangle$ with
\[
\phi_j(x)=\sum_{v\in\sV}\kk(x,\alpha(v))\,e_{s(v)},
\qquad
\psi_j(y)=\sum_{v\in\sV}\kk(y,\alpha(v))\,e_{t(v)},
\]
noting that $s(\cdot),t(\cdot)$, and hence $\phi_j,\psi_j$, all depend on $j$.
Using the properties of $\kk$ we compute
\begin{align*}
\beta_j(x,y)
&=\langle\phi_j(x),\psi_j(y)\rangle\\
&=\Bigl\langle\sum_{v\in\sV}\kk(x,\alpha(v))\,e_{s(v)},\,
\sum_{v'\in\sV}\kk(y,\alpha(v'))\,e_{t(v')}\Bigr\rangle\\
&=\sum_{v\in\sV}\sum_{v'\in\sV}\kk(x,\alpha(v))\,\kk(y,\alpha(v'))\,
\langle e_{s(v)},e_{t(v')}\rangle\\
&=\sum_{(v,v')\in\sE^{(j)}}\kk(x,\alpha(v))\,\kk(y,\alpha(v')),
\end{align*}
where $\sE^{(j)}$ is the set of edges $(v,v')\in\sE$ whose realizing hub path has
color $j$. The last equality holds because
$\langle e_{s(v)},e_{t(v')}\rangle=1$ if and only if $s(v)=t(v')$, and
$s(v)=t(v')$ if and only if there is a hub $h$ to which $v$ has a color-$j$
outgoing edge and from which $v'$ has a color-$j$ incoming edge; this is exactly
the condition for $(v,v')$ to be realized by a monochromatic (color-$j$) path in
the hub extension, i.e.\ $(v,v')\in\sE^{(j)}$.

Each product $\kk(x,\alpha(v))\kk(y,\alpha(v'))$ is supported in
$B_r(\alpha(v))\times B_r(\alpha(v'))$; since the balls $B_r(\alpha(w))$,
$w\in\sV$, are pairwise disjoint (admissibility, with $r<r_0$), near
$(\alpha(v),\alpha(v'))$ every term of the sum except the one indexed by
$(v,v')$ vanishes, so $\beta_j$ coincides locally with the single product
$\kk(x,\alpha(v))\kk(y,\alpha(v'))$. The latter attains a local maximum at
$(\alpha(v),\alpha(v'))$ and is nonconstant nearby; hence $\beta_j$ attains a
local maximum at $(\alpha(v),\alpha(v'))$ for every color-$j$ edge $(v,v')$.
Since every edge of $\sE$ is realized by some hub, and thus receives that hub's
color, the functions $\beta_1,\ldots,\beta_h$ form a density cover of
$(\sV,\sE)$, with $\operatorname{rank}\beta_j\le|H|=n$. Finally,
$u_j=\beta_j\,a_\theta\in\mathcal I$.
\end{proof}

\section{Standard Transformer Models}

We now recover the standard transformer as a special case. Although the
transformer is implemented on $\R^{d}$, the presence of normalization layers
(and the use of cosine similarity) makes it natural to regard it as modeled on a
high-dimensional sphere; the reader is referred to the discussion of the
interacting-particle system and the sphere formulation in Geshkovski et al.\
\cite{TR}.

A simple interacting-particle system with linear attention is
\[
\frac{dz_i}{dt}=\sum_{j=1}^i\langle Kz_j(t),Qz_i(t)\rangle\;Vz_j(t),
\]
with softmax version
\[
\frac{dz_i}{dt}=\sum_{j=1}^i\frac{1}{\mathcal Z_i}
\exp\langle Kz_j(t),Qz_i(t)\rangle\;Vz_j(t),
\qquad
\mathcal Z_i=\sum_{k=1}^i\exp\langle Kz_k(t),Qz_i(t)\rangle,
\]
where $Q,K,V$ are the query, key, and value matrices. The value matrix is taken
antisymmetric, so that it defines a vector field tangent to the sphere
(otherwise one projects onto the tangent space): if $V^{\top}=-V$ then
$\langle Vx,x\rangle=0$, so $Vx\in T_xS^{d-1}$. Comparing with the proof of
\cref{thm:main}, here $\phi(x)=Kx$ and $\psi(y)=Qy$; the parameter space
$\Theta$ is the space of matrices, with $\theta\in\Theta$ the value matrix
$V$, i.e.\ $a_\theta(x,y)=\bv(\theta)_y=\theta y$. The same formulation applies,
with almost no change, on the sphere, where the Euclidean fields are projected
to the tangent space.

The standard formulation suggests the kernel $\kk(x,y)=x\cdot y$. This linear
kernel is comparable to the one used in \cref{thm:main}---on the unit sphere it
is maximal on the diagonal and radially decreasing in the geodesic
distance---but it does not satisfy all the properties required there (it is
neither nonnegative, nor compactly supported, nor zero on edges), so the
discussion below is an analogy rather than a literal instance of the theorem. It
then suggests
\[
\phi(x)=\sum_{v\in\sV}(x\cdot\alpha(v))\,e_{s(v)},
\qquad
\psi(y)=\sum_{v\in\sV}(y\cdot\alpha(v))\,e_{t(v)};
\]
in other words the key and query matrices are
\[
K=\sum_{v\in\sV}e_{s(v)}\,\alpha(v)^{\top},
\qquad
Q=\sum_{v\in\sV}e_{t(v)}\,\alpha(v)^{\top},
\]
which are $|H|\times d$ matrices. The standard basis vectors $e_k$ are used only
for convenience; the construction works equally with any choice of
$u_k\in\R^{|H|}$, where $|H|$ is the number of hub nodes. In particular
$\operatorname{rank}K,\operatorname{rank}Q\le|H|$: although $K$ and $Q$ act on the
high-dimensional ambient space, they factor through the $|H|$-dimensional hub
space---exactly the efficient parametrization promised in the introduction.

We contend that hub extensions play a crucial, hidden role in the
key--query--value mechanism. Heuristically, a node $a\in\sV$ contributes a key to
a hub $p\in H$ through an edge $(a,p)\in\widetilde\sE$, while another node
$b\in\sV$ queries the same hub through an edge $(p,b)\in\widetilde\sE$. Suppose,
for instance, that the key matrix is fixed to the identity and only the query and
value matrices vary. Our constructions show that in this case one must color the
original attention digraph rather than a hub extension. Since the chromatic index
of $(\sV,\sE)$ can be far larger than its hub-chromatic index, realizing a given
attention graph in this setting can require a dramatically larger range
dimension for the query matrix.\\

\textbf{Remark.} Consider the case where the key matrix is fixed to $K=\mathrm{Id}$, so that $\phi(x)=x$ (see the proof of main theorem). Then
every node can have its own source hub ($s(v)=v$), so no two nodes share a source
hub and the source side of the hub extension collapses onto $\sV$ itself.
Realizing the attention digraph $(\sV,\sE)$ then amounts to a \emph{proper
edge-coloring} of $(\sV,\sE)$, that is, a coloring of its arcs in which any two
arcs with a common source, and any two arcs with a common target, receive
distinct colors. Writing $\chi'(\sV,\sE)$ for the least number of colors in such
an edge-coloring (equivalently, the chromatic number of the conflict graph on
$\sE$), one has
\[
\chibp(\sV,\sE)\le\chi'(\sV,\sE),
\]
since the hub extension with one hub per arc realizes exactly this
edge-coloring, whereas $\chibp$ is the minimum over \emph{all} hub extensions.
The inequality can be strict, with a large gap: an extension routing many arcs
through a shared hub may use far fewer colors than $\chi'$. Consequently, fixing
$K=\mathrm{Id}$ forgoes hub sharing and can force a number of colors---and hence
a range dimension for $Q$---dramatically larger than $\chibp(\sV,\sE)$.

\section*{Statements and Declarations}

\noindent\textbf{Competing Interests.} The author declares that there are no
financial or non-financial interests that are directly or indirectly related to
the work submitted for publication.

\noindent\textbf{Funding.} [State any funding sources and grant numbers here, or
state that no funding was received for conducting this study.]

\noindent\textbf{Data Availability.} [Not applicable, or state as appropriate;
this is a theoretical study and no datasets were generated or analysed.]

\end{document}